\documentclass{amsart}
\usepackage[T1]{fontenc}
\usepackage[english]{babel}
\usepackage{bbm}
\usepackage{stmaryrd}
\usepackage{amsthm,amsmath,amssymb,amscd}
\usepackage[all,cmtip]{xy}

\theoremstyle{plain}
\newtheorem{teo}{Theorem }[section]
\newtheorem{coro}[teo]{Corollary}
\newtheorem{prop}[teo]{Proposition}
\newtheorem{lem}[teo]{Lemma}

\theoremstyle{definition}

\newtheorem{defi}[teo]{Definition}
\newtheorem{rem}[teo]{Remark}
\newtheorem{note}[teo]{Notation}

\let\oldmathbbm\mathbbm
\def\bun{\oldmathbbm{1}}
\let\mathbbm\mathbb
\let\hat\widehat
\let\subsection\Subsection
\let\subsubsection\Subsubsection
\makeatletter

\DeclareRobustCommand{\pointrait}{\unskip\@addpunct{.}%
\mbox{}\ignorespaces}
\def\th@plain{%
  \let\thm@indent\noindent
  \thm@headfont{\bfseries\upshape}%
  \thm@notefont{\bfseries\upshape}%
  \thm@preskip.5\linespacing \@plus .5\linespacing
  \thm@postskip\thm@preskip
  \thm@headpunct{\bfseries\pointrait}
  \itshape }
\makeatother

\begin{document}

\title[Rigidity index preservation of regular holonomic D-modules]{Rigidity index preservation of regular holonomic D-modules under Fourier transform}
       
\author[A. Paiva]{A. Paiva}
\address{CMAF, Universidade de Lisboa, Av. Prof. Gama Pinto 2, 1649-003 Lisboa, Portugal.}
\email{paiva@cii.fc.ul.pt}

\maketitle

\begin{abstract}
This paper shows algebraically that the Fourier transform preserves the rigidity index of irreducible regular holonomic
$\mathcal{D}_{\mathbb{P}^1}[*\{\infty\}]$-modules.
\end{abstract}

\section{Introduction}

Riemann showed in 1857 that the local system of the hypergeometric equation can be reconstructed up to isomorphism
from the knowledge of the local monodromies around its singular points 0, 1 and $\infty$, by analytic continuation
of the solutions of the equation around the singular points. In modern terminology, local systems on a
projective smooth connected curve $X$ over $\mathbb{C}$, with singularities on a nonempty finite subset of
$X$, satisfying the aforementioned condition are called physically rigid. In \cite{katz} Katz gave necessary and
sufficient conditions for physical rigidity of local systems on the Riemann sphere, based upon on a cohomological
numerical index. He  showed that in characteristic $p > 0$, the Fourier transform preserves this index when the
local system is a perverse sheaf that does not have punctual support nor does its Fourier transform (cf. \cite{katz}
Theorem 3.0.2). Moreover he conjectured that "it should be true that Fourier transform preserves the index of
rigidity in the $\mathcal{D}$-module context" (cf. \cite{katz}, p. 10). This conjecture was proved by S. Bloch and
H. Esnault in \cite{B-E03}. A different proof is given in this paper, when the $\mathcal{D}$-module is regular
holonomic localized at infinity (Theorem \ref{sjcblJSBNCLKJSBC}).

The paper is divided into five sections. The first section reviews some results on rigidity. The second
extends the notion of rigid local systems to the context of holonomic $\mathcal{D}_{\mathbb{P}^1}$-modules.
The third recalls the notion of Fourier transform and computes the rigidity index of the Fourier
transform of irreducible regular holonomic $\mathcal{D}_{\mathbb{P}^1}[*\{\infty\}]$-modules. The fourth translates the
germs of holonomic $\mathcal{D}$-modules on the equivalent category of pairs of vector spaces.
These equivalences are used in last section to show the preservation of the rigidity index referred above 
(Theorem \ref{sjcblJSBNCLKJSBC}). General references for this paper are \cite{bm, c, ms1, le}.

\section{Rigidity index}

In \cite{katz} Katz gives the following necessary and sufficient condition for the physical rigidity
of local systems on $\mathbb{P}^1$.

\begin{teo}\emph{(\cite{katz}, Theorem 1.1.2)}\label{jsnljAsi:ZNXakjs}
Let $\Sigma$ be a non empty finite subset of $\mathbb{P}^1$, $U\doteq\mathbb{P}^1\setminus\Sigma, j:U^{an}\hookrightarrow(\mathbb{P}^1)^{an}$ the open inclusion and $\mathcal{L}$ an irreducible local system
on $U^{an}$ of rank $n\ge 1$. Then $\mathcal{L}$ is physically rigid if and only if
$\chi((\mathbb{P}^1)^{an},j_*\mathcal{E}nd(\mathcal{L}))=
(2-k)n^2+\sum_i\dim\mathrm{Z}(\mathrm{A}_i)=2$, where $\mathrm{Z}(\mathrm{A}_i)\doteq\{A\in\mathrm{End}_{\mathbb{C}}(\mathcal{L}_{s_{i}})|\mathrm{AA}_i=\mathrm{A}_{i}\mathrm{A}\}$,
$k+1=\#\Sigma$ and $\mathrm{A}_i$ is the monodromy of $\mathcal{L}$ at the point $s_i\in\Sigma$.
\end{teo}

\noindent In order to extend the notion of rigidity to the context of holonomic $\mathcal{D}_{\mathbb{P}^1}$-modules,
one start to recall the notion of minimal extension.

\begin{defi}\label{def-ext-min}
Let $\mathcal{M}$ be a holonomic $\mathcal{D}$-module on a Riemann surface $X$ and $\Sigma\subset X$ a finite set. One says
that a holonomic $\mathcal{D}$-module $\mathcal{N}$ on $X$ is a \textbf{minimal extension} of $\mathcal{M}$ along $\Sigma$ and denote it $\mathcal{M}_{min}$ if:
\begin{itemize}
\item [{\romannumeral 1})] $\mathcal{O}_{\mathrm{X}}[*\Sigma]\otimes_{\mathcal{O}_{\mathrm{X}}}\mathcal{M}=\mathcal{O}_{\mathrm{X}}[*\Sigma]\otimes_{\mathcal{O}_{\mathrm{X}}}\mathcal{N}$,
\item [{\romannumeral 2})] $\mathcal{M}$ has neither nonzero submodules nor nonzero quotients with support on a subset of $\Sigma$.
\end{itemize}
\end{defi}

\noindent It is well known that if $\mathcal{M}$ is a regular holonomic $\mathcal{D}_X$-module
 with singularities on $\Sigma$, its local system 
$\mathcal{L}\doteq\mathcal{H}om_{\mathcal{D}_X}(\mathcal{O}_X, \mathcal{M})_{\mid X\setminus\Sigma}$
satisfies the following identity
\begin{equation*}
\mathrm{DR}\left(\left(\mathcal{E}nd_{\mathcal{O}_X}(\mathcal{M}[*\Sigma])\right)_{min}\right)=
j_*\mathcal{E}nd\left(\mathcal{L}\right).
\end{equation*}

\noindent Taking this into account and Theorem \ref{jsnljAsi:ZNXakjs}, one is led to give the following definition.

\begin{defi}[Rigidity index]\label{IndRig}
Let $\mathcal{M}$ be an irreducible holonomic $\mathcal{D}_{\mathbb{P}^1}$-module and
$\Sigma$ the set of its singular points. Set $\mathrm{rig}(\mathcal{M})$ the 
invariant $$\chi\left(\mathbb{P}^1,\mathrm{DR}\left(\left(\mathcal{E}nd_{\mathcal{O}_{\mathbb{P}^1}}(\mathcal{M}[*\Sigma])\right)_{min}\right)\right)$$
\noindent and call it the \textbf{rigidity index} of $\mathcal{M}$.
\end{defi}

Follows three propositions on minimal extension which will be used on this paper.

\begin{prop}
\emph{(\cite{paiva}, Corollary 2.7.4)}
\label{min:form:conv}
If $\mathcal{M}$ is a holonomic $\mathcal{D}$-module on a Riemann surface $X$ and $\Sigma\subset X$ a finite set, then the $\mathcal{D}$-modules $\mathcal{M}$ and $\mathcal{M}[*\Sigma]$ have the same minimal extension along $\Sigma$, i.e. $\mathcal{M}_{min}=\mathcal{M}[*\Sigma]_{min}$.
\end{prop}

\begin{prop}
\emph{(\cite{paiva}, Theorem 2.7.6)}
\label{min}
Let $\mathcal{M}$ be a holonomic $\mathcal{D}$-module on a Riemann surface $X$ and $\Sigma\subset X$ a finite set,
then the minimal extension of $\mathcal{M}$ along $\Sigma$ exists and is given by
$$\mathcal{M}_{min}=\left(\left(\mathcal{M}/\mathrm{H}_{\left[\Sigma\right]}(\mathcal{M})\right)^*\bigl/
\mathrm{H}_{\left[\Sigma\right]}\left(\left(\mathcal{M}/\mathrm{H}_{\left[\Sigma\right]}(\mathcal{M})\right)^*\right)\right)^*,$$
\noindent where $\mathrm{H}_{\left[\Sigma\right]}(\mathcal{M})(U)\doteq\left\{s\in\mathcal{M}(U)\mid\exists x_i\in\Sigma \exists k_i\in\mathbb{N}:(x-x_i)^{k_i}s=0\right\}$, for each \linebreak open set $U\varsubsetneq\mathbb{P}^1$ and local coordinate $x$ on $U$.
\end{prop}

\noindent As for holonomic $\mathcal{D}$ -modules on a Riemann surface $X$ with singularities on $\Sigma$
$\mathrm{supp}(\mathcal{M})\subset\Sigma$, iff $\mathcal{M}$ coincides with its algebraic support on $\Sigma$, i.e. $\mathcal{M}=\mathrm{H}_{[\Sigma]}(\mathcal{M})$, cf. \cite{paiva} Lemma 2.7.8, the notion of minimal extension at
the level of germs is defined as follows.

\begin{defi}
Let $\mathcal{M}$ be a holonomic $\mathcal{D}_x\doteq\mathbb{C}\{x\}\langle\partial_x\rangle$ (resp. 
$\hat{\mathcal{D}}_x\doteq\mathbb{C}[\![x]\!]\langle\partial_x\rangle$) -module. One says that a holonomic
$\mathcal{D}_x$ (resp. $\hat{\mathcal{D}}_x$) -module $\mathcal{N}$ is a \textbf{minimal extension} of 
$\mathcal{M}$ and denote it $\mathcal{M}_{min}$ if:
\begin{itemize}
\item [{\romannumeral 1})] $\mathcal{M}[x^{-1}]=\mathcal{N}[x^{-1}]$,
\item [{\romannumeral 2})] $\mathcal{M}$ has neither nonzero submodules nor nonzero quotients coinciding with its algebraic support on $0$.
\end{itemize}
\end{defi}

\begin{prop}
\emph{(\cite{paiva}, Theorem 2.7.11)}
\label{formal:min:com}
The minimal extension commutes with the formalized, that is, if $\mathcal{M}$ is a holonomic $\mathcal{D}_x$-module, then
$\widehat{\mathcal{M}_{min}}\simeq(\hat{\mathcal{M}})_{min}$.
\end{prop}

\section{Fourier transform}

The notion of Fourier transform is built on the concept of {\it twisted modules}, cf. \cite{c} p. 38.
Let us recall this notion. Let $R$ be a ring, $M$ a left $R$-module and $\sigma$ an automorphism of $R$.
$M_{\sigma}$ is the left $R$-module $M$ with the new action $a\bullet m\doteq\sigma(a)m$, for 
$a\in R$ and $m\in M$. A routine
calculation shows that $M_{\sigma}$ is a left $R$-module and $\sigma$ defines a functor from the category of left
$R$-modules into itself. $M_{\sigma}$ is called the 
{\it twisted module} of $M$ by $\sigma$.
Let us apply this construction to $A_1\doteq\mathbb{C}[x]\langle\partial_x\rangle$ (resp.\ $A_1[x^{-1}]$
or $\mathcal{D}$ or $\hat{\mathcal{D}}$) to define the Fourier transform (resp.\ the inversion).

\begin{defi} The Fourier transform and the inversion are, respectively, the following automorphisms:
$$
\begin{array}{rclcrcl}
\mathcal{F}: \mathbb{C}[x]\langle\partial_x\rangle & \longrightarrow & \mathbb{C}[x]\langle\partial_x\rangle, & &
\mathcal{I}: \mathbb{C}[x, x^{-1}]\langle\partial_x\rangle & \longrightarrow & \mathbb{C}[x, x^{-1}]\langle\partial_x\rangle. \\
x & \longmapsto &-\partial_x & & x & \longmapsto & x^{-1}\\
\partial_x & \longmapsto & x & & \partial_x & \longmapsto &-x^2\partial_x\\
\end{array}
$$
\end{defi}

To extend the notion of Fourier transform from the context of holonomic $A_1$-mod\-ules
to the context of holonomic
$\mathcal{D}_{\mathbb{P}^1}[*\{\infty\}]\doteq\mathcal{D}_{\mathbb{P}^1}^{\mathrm{ana}}[*\{\infty\}]$-modules, let~us~re\-call
the correspondence between holonomic $A_1$, $\mathcal{D}_{\mathbb{P}^1}^{\mathrm{alg}}[*\{\infty\}]$ and
$\mathcal{D}_{\mathbb{P}^1}^{\mathrm{ana}}[*\{\infty\}]$-mod\-ules.
Let $\mathrm{M}$ be a holonomic $A_1$-module with singularities on
$\mathrm{S}=\{\gamma_1,\dots,\gamma_k\}\subset\mathbb{C}$, where $k\ge 1$.
Let $\mathrm{M}'$ be the holonomic $A_1[x^{-1}]$-module $\mathrm{M}[x^{-1}]_{\mathcal{I}}$.
This two modules generate the holonomic $\mathcal{D}_{\mathbb{P}^1}^{\mathrm{alg}}[*\{\infty\}]$-mod\-ule
$\mathcal{M}^{\mathrm{alg}}$ defined by $\mathcal{M}^{\mathrm{alg}}(\mathbb{C})=\mathrm{M}$ and 
$\mathcal{M}^{\mathrm{alg}}(\mathbb{P}^1\setminus\{\infty\})=\mathrm{M}'$ and vice versa. The
correspondence between $\mathcal{M}^{\mathrm{ana}}\doteq\mathcal{M}^{\mathrm{alg}}\otimes_{\mathcal{O}_{\mathbb{P}^1}^{\mathrm{alg}}}\mathcal{O}_{\mathbb{P}^1}^{\mathrm{ana}}$ and $\mathcal{M}^{\mathrm{alg}}$ is done by GAGA, cf. \cite{mal} chap. I \S 4.
The holonomic $\mathcal{D}_{\mathbb{P}^1}[*\{\infty\}]$-module ${\mathcal{M}}_{\mathcal{F}}$
built this way from ${\mathrm{M}}_{\mathcal{F}}$ is called the {\it Fourier transform} of $\mathcal{M}\doteq\mathcal{M}^{\mathrm{ana}}$.
In the special case when all singularities $\Sigma\doteq\{\gamma_0=\infty\}\cup\mathrm{S}$ of $\mathcal{M}$
are regular, ${\mathcal{M}}_{\mathcal{F}}$ has a regular
singularity at $0$ and one (possibly irregular) at $\infty$, cf. \cite{mal} chap. V \S 1.
Now follows the notation used to compute $rig({\mathcal{M}}_{\mathcal{F}})$.

\begin{note}Let $\mathcal{M}$ be a regular holonomic
$\mathcal{D}_{\mathbb{P}^1}$-module with singularities on~$\Sigma$.
\begin{itemize}
\item [{\romannumeral 1})] $(\hat{\mathcal{N}},\hat{\nabla})$ denotes the formalized at $\infty$ of 
$(\mathcal{N},\nabla)$ of the usual \mbox{meromorphic}~con\-nexion equivalent to
${\mathcal{M}}_{\mathcal{F}}[*\{0,\infty\}]$, which has the decomposition of Turrittin
\begin{equation}\label{t-dec}
(\hat{\mathcal{N}},\hat{\nabla})\simeq\bigoplus_{i=1}^k\hat{\mathcal{E}}^{\varphi_i}
\otimes(\hat{\mathrm{R}}_i,\hat{\nabla}_i)
\end{equation}
\noindent and $(\hat{\mathrm{R}}_i,\hat{\nabla}_i)$ are regular meromorphic connexions,
cf. Turrittin \cite{turritin}, Levelt~\cite{levelt}, and \cite{paiva} Theorem 1.9.5 and Lemma 1.96;
\item [{\romannumeral 2})] $\mathrm{T}_i$ denotes the monodromy of
$(\hat{\mathrm{R}}_i,\hat{\nabla}_i)$ and $n_i$ the dimension of $\hat{\mathrm{R}}_i;$
\item [{\romannumeral 3})] $\mathrm{T}$ denotes the monodromy at 0 of the local system
$\mathcal{H}om_{\mathcal{D}_{\mathbb{P}^1}}(\mathcal{O}_{\mathbb{P}^1}, {\mathcal{M}}_{\mathcal{F}})|_{\mathbb{C}^*}$.
\end{itemize}
\end{note}

\begin{teo}\label{rig:transforme:fourier}
If $\mathcal{M}$ is a regular holonomic
$\mathcal{D}_{\mathbb{P}^1}[*\{\infty\}]$-module with singularities on $\Sigma$, the rigidity index of ${\mathcal{M}}_{\mathcal{F}}$ is given by
\begin{equation}
rig({\mathcal{M}}_{\mathcal{F}})=\dim\mathrm{Z}(\mathrm{T})+\sum_{i=1}^k
\dim\mathrm{Z}(\mathrm{T}_i) +\sum_{i=1}^kn_i^2-\bigg(\sum_{i=1}^kn_i\bigg)^2.
\label{rig-geral0}
\end{equation}
\end{teo}

\begin{proof}
Let $\mathcal{E}$ be the $\mathcal{D}_{\mathbb{P}^1}$-module $\mathcal{E}nd_{\mathcal{O}_{\mathbb{P}^1}}\left({\mathcal{M}}_{\mathcal{F}}[*\{0,\infty\}]\right)$,
$\mathcal{F}^{\bullet}$
the de Rham complex $\mathrm{DR}\left(\mathcal{E}_{min}\right)$ on $\mathbb{P}^1$
and $j\colon\mathbb{C}^*\hookrightarrow\mathbb{P}^1$ the open inclusion.
One has the short exact sequence
\begin{equation}
0\longrightarrow j_!j^{-1}\mathcal{F}^{\bullet}\stackrel{\eta}{\longrightarrow}
\mathcal{F}^{\bullet}\longrightarrow\mathrm{coker}\:\eta\longrightarrow 0,
\label{0-pervers-try-lem-0}
\end{equation}
\noindent which yields the identity:
$$\chi(\mathbb{P}^1,\mathcal{F}^{\bullet})=
\chi(\mathbb{P}^1,j_!j^{-1}\mathcal{F}^{\bullet})+
\chi(\mathbb{P}^1,\mathrm{coker}\:\eta).$$
\noindent Set
$\mathcal{L}=h^0(j^{-1}\mathcal{F}^{\bullet})$, hence:
\begin{eqnarray*}
&\chi(\mathbb{P}^1,j_!j^{-1}\mathcal{F}^{\bullet})=(2-2)\mathrm{rang}\:\mathcal{L}=0,&\\
&\chi(\mathbb{P}^1,\mathrm{coker}\:\eta)=
\chi(\mathbb{P}^1,(\mathrm{coker}\:\eta)_0)+
\chi(\mathbb{P}^1,(\mathrm{coker}\:\eta)_{\infty}),&
\end{eqnarray*}
\noindent because $\{0,\infty\}$ are the only singularities of
$\mathcal{F}^{\bullet}$, thus:
\begin{equation}
\chi(\mathbb{P}^1,\mathcal{F}^{\bullet})=
\chi(\mathbb{P}^1,(\mathrm{coker}\:\eta)_0)+
\chi(\mathbb{P}^1,(\mathrm{coker}\:\eta)_{\infty}).\label{rig-geral1}
\end{equation}

To compute $\chi(\mathbb{P}^1,(\mathrm{coker}\:\eta)_0)$ one
takes a disk $\mathrm{D}\subset\mathbb{C}$ centered at 0 and the inclusion
$i:\mathrm{D}^*\hookrightarrow\mathrm{D}$. Set
$\mathcal{G}^{\bullet}=\mathrm{DR}\left(\mathcal{E}_{min}|_{\mathrm{D}}\right)$.
As $\mathcal{E}_{min}|_{\mathrm{D}}$ is a regular holonomic $\mathcal{D}_{\mathrm{D}}$
-module on D,
$\mathrm{DR}\left(\mathcal{E}_{min}|_{\mathrm{D}}\right)=i_*\mathcal{E}nd(\mathcal{L}^{\prime})$,
where
$\mathcal{L}^{\prime}=\mathcal{H}om_{\mathcal{D}_{\mathbb{P}^1}}\left(\mathcal{O}_{\mathbb{P}^1},{\mathcal{M}}_{\mathcal{F}}\right)|_{{\mathrm{D}}^*}$.
$\mathcal{G}^{\bullet}$ is a perverse complex on D and gives rise to the short exact sequence
\begin{equation}
0\longrightarrow i_!i^{-1}\mathcal{G}^{\bullet}\stackrel{\eta|_{\mathrm{D}}}{\longrightarrow}\mathcal{G}^{\bullet}\longrightarrow\left(\mathrm{coker}\:\eta\right)|_{\mathrm{D}}\longrightarrow 0, \label{0-pervers-try-lem-2}
\end{equation}
\noindent which yields
$$\chi(\mathrm{D},\mathcal{G}^{\bullet})=\chi(\mathrm{D},i_!i^{-1}\mathcal{G}^{\bullet})+\chi(\mathrm{D},\mathrm{coker}\:\eta|_{\mathrm{D}}).$$
\noindent Since
$h^0(i^{-1}\mathcal{G}^{\bullet})=\mathcal{E}nd(\mathcal{L}^{\prime})$, one has
$\chi(\mathrm{D},i_!i^{-1}\mathcal{G}^{\bullet})=(2-2)\mathrm{rang}\:\mathcal{E}\mathrm{nd}(\mathcal{L}^{\prime})$ $=0$,
\noindent because $\mathrm{D}^*$ is homotopic to $\mathrm{S}^1$, therefore
$\chi(\mathrm{D},\mathcal{G}^{\bullet})=\chi(\mathrm{D},(\mathrm{coker}\:\eta)_0)$.
The exact sequence (\ref{0-pervers-try-lem-2}) implies that
$(\mathcal{G}^{\bullet})_0=(\mathrm{coker}\:\eta)_0$.
Moreover
$(\mathcal{G}^{\bullet})_0=\{\mathrm{M}\in\mathrm{End}(\mathrm{E})\mid\mathrm{T}_{\mathcal{E}nd(\mathcal{L}^{\prime})}(\mathrm{M})=\mathrm{M}\}$,
where $\mathrm{E}=h^0(\mathrm{D}\setminus\mathbb{R}^+,\mathcal{L}^{\prime})$,
$\mathrm{T}_{\mathcal{E}nd(\mathcal{L}^{\prime})}=ad_{\mathrm{T}}$ and T is the
monodromy of $\mathcal{L}^{\prime}$ at 0, thus
$(\mathcal{G}^{\bullet})_0=\{\mathrm{M}\in\mathrm{End}(\mathrm{E})\mid\mathrm{TM}=\mathrm{MT}\}\doteq\mathrm{Z}(\mathrm{T})$.
By Mayer-Vietoris one shows that
$\chi(\mathbb{P}^1,(\mathrm{coker}\:\eta)_0)=\chi(\mathrm{D},\mathrm{coker}\:\eta|_{\mathrm{D}})$,
hence
\begin{equation}
\chi(\mathbb{P}^1,(\mathrm{coker}\:\eta)_0)=\dim\mathrm{Z}(\mathrm{T}).
\label{rig-geral2}
\end{equation}\
\noindent Now one computes $\chi(\mathbb{P}^1,(\mathrm{coker}\:\eta)_{\infty})$.
As $\left(j_!j^{-1}\mathcal{F}^{\bullet}\right)_{\infty}=0$,
the exact sequence (\ref{0-pervers-try-lem-0}) implies that
$\left(\mathcal{F}^{\bullet}\right)_{\infty}\simeq\left(\mathrm{coker}\:\eta\right)_{\infty}$.
On the other hand
$\chi\left(\mathbb{P}^1,\left(\mathrm{DR}\left(\mathcal{E}_{min}\right)\right)_{\infty}\right)=\chi\left(\left(\mathcal{E}_{min}\right)_{\infty},\left(\mathcal{O}_{\mathbb{P}^1}\right)_{\infty}\right)$,
therefore by definition of irregularity
\begin{equation}
\chi\left(\mathbb{P}^1,\left(\mathrm{DR}\left(\mathcal{E}_{min}\right)\right)_{\infty}\right)=
\chi(\widehat{(\mathcal{E}_{min})}_{\infty},(\hat{\mathcal{O}}_{\mathbb{P}^1})_{\infty})-
i\left((\mathcal{E}_{min})_{\infty}\right).\label{rig-geral3}
\end{equation}
\noindent Owing to Lemma \ref{Tourritin-monodromy} and Proposition \ref{formal:min:com}
\begin{equation}
\chi(\widehat{(\mathcal{E}_{min})}_{\infty},
(\hat{\mathcal{O}}_{\mathbb{P}^1})_{\infty})=
\sum_{i=1}^k\dim\mathrm{Z}(\mathrm{T}_i).\label{rig-geral4}
\end{equation}
\noindent Furthermore, by Lemma \ref{Tourritin-monodromy}
\begin{equation}
i(\left(\mathcal{E}_{min}\right)_{\infty})=\bigg(\sum_{i=1}^kn_i\bigg)^2-\sum_{i=1}^kn_i^2.
\label{rig-geral5}
\end{equation}
\noindent The identity (\ref{rig-geral0}) is now an immediate consequence of
(\ref{rig-geral1}), (\ref{rig-geral2}), (\ref{rig-geral3}), (\ref{rig-geral4}) and (\ref{rig-geral5}).
\end{proof}

\begin{lem}\label{Tourritin-monodromy}
Let $(\hat{\mathcal{N}},\hat{\nabla})$ be the formalized at $\infty$ of the (usual) meromorphic connexion associated
to the Fourier transform of a regular holonomic $\mathcal{D}_{\mathbb{P}^1}[*\{\infty\}]$-module with singularities on $\Sigma$.
Assume that $\varphi_1=0$ on the Turrittin decomposition of $(\hat{\mathcal{N}},\hat{\nabla})$, notice that $(\hat{\mathrm{R}}_1,\hat{\nabla}_1)$ might be 0, then:
\begin{itemize}
\item [{\romannumeral 1})] $\chi(\hat{\mathcal{N}}_{min},\mathbb{C}[\![x]\!])=\dim\{\mathrm{e}\mid\mathrm{T}_1\mathrm{e}=\mathrm{e}\}$,
\item [{\romannumeral 2})] $\chi(\mathcal{E}nd_{\hat{\mathcal{O}_x}}(\hat{\mathcal{N}})_{min},\mathbb{C}[\![x]\!])=\displaystyle\sum_{i=1}^k\dim\mathrm{Z}(\mathrm{T}_i)$,
\item [{\romannumeral 3})] $i\left(\mathcal{E}nd_{\mathcal{O}_x}(\mathcal{N})\right)=\left(\displaystyle\sum_{i=1}^kn_i\right)^2-\displaystyle\sum_{i=1}^kn_i^2$.
\end{itemize}
\end{lem}
\begin{proof}
{\romannumeral 1}) For each term from decomposition (\ref{t-dec}) either $i=1$ or $i>1$.

If $i>1$, the holonomic
$\hat{\mathcal{D}}_x$-module
$\hat{\mathcal{N}}_i\doteq\hat{\mathcal{E}}^{\varphi_i}\otimes(\hat{\mathrm{R}}_i,\hat{\nabla}_i)$
has no regular component, therefore
$\hat{\mathcal{N}}_i\simeq\hat{\mathcal{N}}_i[x^{-1}]$
(cf. \cite{ms1} Theorem 6.3.1). This implies that the multiplication by $x$ is bijective so $\mathrm{H}_{[0]}(\hat{\mathcal{N}}_i)=0$ and $\mathrm{H}_{[0]}(\hat{\mathcal{N}}_i^*)=0$, hence
$(\hat{\mathcal{N}}_i)_{min}=\hat{\mathcal{N}}_i[x^{-1}]$ by Proposition \ref{min}.
Thanks to these isomorphisms one has
$\chi(\hat{\mathcal{N}}_{min},\mathbb{C}[\![x]\!])=\chi((\hat{\mathcal{N}}_1)_{min},\mathbb{C}[\![x]\!])$,
because $\chi(\hat{\mathcal{N}}_i[x^{-1}],\mathbb{C}[\![x]\!])=0$ for $i>1$.

If $i=1$, choosing a base and a coordinate system one has the isomorphism
$$(\hat{\mathrm{R}}_1,\hat{\nabla}_1)\simeq\left(\mathbb{C}[\![x]\!]^{n},x\frac{d\;\;}{dx}-\mathrm{A}_1\right), n=\dim\hat{\mathrm{R}}_1.$$
\noindent By Proposition \ref{min:form:conv}, $\hat{\mathcal{N}}_1$ and
$\hat{\mathcal{N}}_1[x^{-1}]$ have the same minimal extension, so 
one computes the minimal extension of the later. To do this, take a meromorphic change of base,
which transforms $\mathrm{A}_1$ in the constant matrix,
$\mathrm{J}$, in the Jordan canonical form
$$(\hat{\mathrm{R}}_1[x^{-1}],\hat{\nabla}_1)\simeq\bigoplus_{j=1}^{m}\left(\mathbb{C}[\![x]\!]^{n_{j}},x\frac{d\;\;}{dx}-\mathrm{J}_{j}\right),$$
\noindent with $n_{1}+\cdots+n_{m}=n$ and $\mathrm{J}_{j}$ the Jordan blocs of $\mathrm{J}$. If $\alpha_{j}$ is the eigenvalue of the Jordan block $\mathrm{J}_{j}$, then one has the isomorphism
$$\hat{\mathcal{N}}_1[x^{-1}]\simeq\bigoplus_{j=1}^{m}\mathbb{C}[\![x]\!][x^{-1}]\langle\partial_x\rangle/\mathbb{C}[\![x]\!][x^{-1}]\langle\partial_x\rangle.(x\partial_x-\alpha_{j})^{n_{j}}.$$

If $\alpha_{j}\not\in\mathbb{Z}$, the Bernstein polynomial $b(x)$ of
$\hat{\mathcal{N}}_{1j}\doteq\mathbb{C}[\![x]\!]\langle\partial_x\rangle/(x\partial_x-\alpha_{j})^{n_{j}}$
is $(x-\alpha_{j})^{n_{j}}$, therefore for each $k\in\mathbb{N}$, $b(k)\neq 0$, thus
the multiplication by $x$
$$\mathbb{C}[\![x]\!]\langle\partial_x\rangle/(x\partial_x-\alpha_{j})^{n_{j}}\stackrel{x.}{\longrightarrow}\mathbb{C}[\![x]\!]\langle\partial_x\rangle/(x\partial_x-\alpha_{j})^{n_{j}}$$
\noindent is a bijective map (cf. \cite{ms1} Lemma 4.2.7), i.e. $\hat{\mathcal{N}}_{1j}$
is a meromorphic connexion. In particular one has
$\mathrm{H}_{[0]}(\hat{\mathcal{N}}_{1j})=0$ and
$\mathrm{H}_{[0]}((\hat{\mathcal{N}}_{1j})^*)=0$,
thus $(\hat{\mathcal{N}}_{1j})_{min}=\hat{\mathcal{N}}_{1j}$. As 
$\chi(\hat{\mathcal{N}}_{1j}[x^{-1}],\mathbb{C}[\![x]\!])=0$ for each
$\alpha_j\not\in\mathbb{Z}$, these equalities lead to:
$$\chi(\hat{\mathcal{N}}_{min},\mathbb{C}[\![x]\!])=\sum_{\alpha_j\in\mathbb{Z}}
\chi((\hat{\mathcal{N}}_{1j})_{min},\mathbb{C}[\![x]\!]).$$

If
$\alpha_{j}\in\mathbb{Z}$, one can assume $\alpha_{j}=0$, because after the change of base
$\mathrm{B}=x^{-\alpha_{j}}\bun $, the matrix $\mathrm{J}_{j}$ is transformed into
$$\mathrm{BJ}_{j}\mathrm{B}^{-1}+x\partial\mathrm{B}/\partial x\mathrm{B}^{-1}=\mathrm{J}_{j}-\alpha_{j}\bun .$$
In this case $\hat{\mathcal{N}}_j=\mathbb{C}[\![x]\!]\langle\partial_x\rangle/(x\partial_x)^{n_j}$ and therefore
$(\hat{\mathcal{N}}_j)_{min}=\mathbb{C}[\![x]\!]\langle\partial_x\rangle/R_j$, where
$R_j=\partial_x(x\partial_x)^{n_j-1}$.
As $\{1,\log x,\dots,\log^{n_{j}-1}x\}$ is a solution base of the differential equation $R_{j}y=0$,
the kernel of $R_{j}$ in $\mathbb{C}[\![x]\!]$ is $\langle1\rangle$, hence $\dim\ker R_{j}=1$.
Given that $\frac{1}{(l+1)^{n_{j}}}x^{l+1}$ is a solution of the differential equation $R_{j}y=x^l$, $\dim\:\mathrm{coker}\:R_{j}=0$.

Altogether $\chi(\hat{\mathcal{N}}_{min},\mathbb{C}[\![x]\!])=\#\{\mathrm{J}_{j}\mid\alpha_{j}\in\mathbb{Z}\}$,
that is $\chi(\hat{\mathcal{N}}_{min},\mathbb{C}[\![x]\!])=\dim\{\mathrm{e}\mid\mathrm{T}_1\mathrm{e}=\mathrm{e}\}$.

\noindent {\romannumeral 2}) Take the decomposition of Turrittin (\ref{t-dec}) $(\hat{\mathcal{N}},\hat{\nabla})\simeq\bigoplus_{i=1}^k\hat{\mathcal{E}}^{\varphi_i}\otimes\hat{\mathrm{R}}_i$. Since:
{\small
\begin{eqnarray}
\Biggl(\mathrm{Hom}_{\hat{\mathcal{O}}_x}\Biggl(\bigoplus_{i=1}^k\hat{\mathcal{E}}^{\varphi_i}\otimes\hat{\mathrm{R}}_i,\!&\!\!\!\!\displaystyle\bigoplus_{i=1}^k\!\!\!\!&\!\hat{\mathcal{E}}^{\varphi_i}\otimes\hat{\mathrm{R}}_i\Biggr),\hat{\nabla}\Biggr)\;\simeq \nonumber \\
& \simeq & \bigoplus_{i=1}^k\bigoplus_{j=1}^k\left(\mathrm{Hom}_{\hat{\mathcal{O}}_x}\left(\hat{\mathcal{E}}^{\varphi_i}\otimes\hat{\mathrm{R}}_i,\hat{\mathcal{E}}^{\varphi_j}\otimes\hat{\mathrm{R}}_j\right),\hat{\nabla}\right)\nonumber \\
& \simeq & \bigoplus_{i=1}^k\bigoplus_{j=1}^k\left(\mathrm{Hom}_{\hat{\mathcal{O}}_x}\left(\hat{\mathrm{R}}_i,\hat{\mathcal{E}}^{\varphi_j-\varphi_i}\otimes\hat{\mathrm{R}}_j\right),\hat{\nabla}\right)\nonumber \\
& \simeq & \bigoplus_{i=1}^k\bigoplus_{j=1}^k\left(\mathrm{Hom}_{\hat{\mathcal{O}}_x}\left(\hat{\mathrm{R}}_i,\hat{\mathrm{R}}_j\right)\otimes\hat{\mathcal{E}}^{\varphi_j-\varphi_i},\hat{\nabla}\right),\label{irr-end}
\end{eqnarray}
}
\noindent the statement {\romannumeral 1}) implies that
\begin{align*}
\chi(\mathrm{End}(\hat{\mathcal{N}})_{min},\mathbb{C}[\![x]\!]) & = \dim\{\mathrm{M}=\mathrm{M}_1\oplus\cdots\oplus\mathrm{M}_k\mid\mathrm{ad}_{\mathrm{T}_1}\oplus\cdots\oplus\mathrm{ad}_{\mathrm{T}_k}\mathrm{M}=\mathrm{M}\}\\
& = \displaystyle\sum_{i=1}^k\dim\mathrm{Z}(\mathrm{T}_i).
\end{align*}
\noindent {\romannumeral 3}) Given the decomposition of Turrittin (\ref{t-dec}) of $(\hat{\mathcal{N}},\hat{\nabla})$, where $(\hat{\mathrm{R}}_i,\hat{\nabla}_i)$ are regular meromorphic connexions, the irregularity of each component is either
 $0$, if
$i=1$, or $n_i=\mathrm{rang}(\hat{\mathrm{R}}_i)$, if $i>1$. Applying the same reasoning to the decomposition of Turrittin of the endomorphisms of $\hat{\mathcal{N}}$, i.e. to the identity (\ref{irr-end}), one has:
\[
i\left(\mathcal{E}nd_{\mathcal{O}}(\mathcal{N})_{min}\right) = \sum_{{i,j=1} \atop {i\neq j}}^k\mathrm{rang}\left(\mathrm{Hom}_{\hat{\mathcal{O}}_x}(\hat{\mathrm{R}}_i,\hat{\mathrm{R}}_j)\right)=\bigg(\sum_{i=1}^kn_i\bigg)^2-\sum_{i=1}^kn_i^2,
\]
\noindent where $n_i=\mathrm{rang}(\mathrm{T}_i)=\mathrm{rang}(\hat{\mathrm{R}}_i)$.
\end{proof}

\section{Pairs of vector spaces}

The previous two sections show that both the rigidity index of a
regular holonomic $\mathcal{D}_{\mathbb{P}^1}[*\{\infty\}]$-module as well as of
its Fourier transform are expressed in terms of the monodromy at its singular points, 
cf. theorems \ref{jsnljAsi:ZNXakjs} and \ref{rig:transforme:fourier}.
Moreover not only the category $\Theta$ of pairs of vector spaces is equivalent
(resp. anti-equivalent) to the category of regular holonomic $\mathcal{D}_x$-modules
(resp. the category of germs of complexes of perverse sheaves), but also the monodromy appears
there naturally.

\begin{teo}
\emph{(\cite{paiva}, Proposition 2.4.5)}
\label{pervers:classification}
Let $\mathcal{F}^{\bullet}$ be a complex of perverse sheaves on $D=B_{\varepsilon}(0)\subset\mathbb{C}$, $\varepsilon>0$, $j:D^*\hookrightarrow D$, $E=h^0(\mathcal{F}^{\bullet}(D\setminus\mathbb{R}^+))$,
$F=(\mathrm{R}^1\Gamma_{\mathbb{R}^+\cap D}\mathcal{F}^{\bullet})_0$ and $T$ the monodromy of the local system $h^0(\mathcal{F}^{\bullet}|_{D^*})$. One has the following representations in the category $\Theta$:
\begin{itemize}
\item [{\romannumeral 1})] If $\mathcal{F}^{\bullet}=j_{!}(h^0(\mathcal{F}^{\bullet}|_{D^*})),$ 
$\mathcal{F}^{\bullet}$ is represented by 
$\xymatrix{ E \ar@/^0.2cm/[r]^{_{\bun _E}} & E, \ar@/^0.2cm/[l]^{^{T\;-\;\bun _E}}}$
where $E=F.$
\item [{\romannumeral 2})] If $\mathcal{F}^{\bullet}=j_{*}(h^0(\mathcal{F}^{\bullet}|_{D^*})),$
$\mathcal{F}^{\bullet}$ is represented by
$\xymatrix{ E \ar@/^0.2cm/@{->>}[r]^{_{T\;-\;\bun _E}} & F, \ar@/^0.2cm/@{_{(}->}[l]^{^{inc}}}$
where $F\subset E.$
\item [{\romannumeral 3})] If $\mathcal{F}^{\bullet}=Rj_{*}(h^0(\mathcal{F}^{\bullet}|_{D^*})),$
$\mathcal{F}^{\bullet}$ is represented by
$\xymatrix{ E \ar@/^0.2cm/[r]^{_{T\;-\;\bun _E}} & E, \ar@/^0.2cm/[l]^{^{\bun _E}}}$
where $E=F.$
\end{itemize}
\end{teo}

\begin{rem}\label{hbbcjbclkjbc}
The pair of vector spaces in statement {\romannumeral 3}) of theorem above is the representant
in the category $\Theta$ of the germs of localized regular holonomic
$\mathcal{D}$-modules (see \cite{ms1} pp. 40 and 41).
\end{rem}

The following Theorem shows that the notion of minimal extension in the category $\Theta$ is meaningful.

\begin{teo}
\emph{(\cite{paiva}, Proposition 2.7.13)}
\label{prop:min:pervers}
Let $D$ be a disk centered at the origin, $j:D^*\hookrightarrow D$ the inclusion
and $\mathcal{L}$ a locally constant sheaf on $D^*$. If $\mathcal{F}^{\bullet}$ is a complex
of perverse sheaves on $D$ such that $\mathcal{F}^{\bullet}|_{D^*}=\mathcal{L},$ then:
\begin{itemize}
\item [{\romannumeral 1})] there exists
$\eta : j_*\mathcal{L}\rightarrow\mathcal{F}^{\bullet}$ morphism of of complexes of perverse
sheaves such that $\eta|_{D^*}=id_{\mathcal{L}}$,
\item [{\romannumeral 2})]$j_*\mathcal{L}$ has neither kernels nor cokernels with support at the origin.
\end{itemize}
Kernels, cokernels are taken in the category of complexes of perverse sheaves sense.
\end{teo}

\begin{defi}[Minimal extension]\label{asdnsjdnja}
Let $D$ be a disk centered at the origin, $j:D^*\hookrightarrow D$ the inclusion,
$\mathcal{F}^{\bullet}$ a complex of perverse sheaves on $D$
and $\mathcal{L}=h^0(\mathcal{F}^{\bullet}|_{D^*})$.
One calls {\bf minimal extension} of $\mathcal{F}^{\bullet}$ to the
complex $j_*\mathcal{L}$.
\end{defi}

\begin{prop}\label{jbcbjkbjzka}
If the pair $\xymatrix@1{\mathrm{E} \ar@/^0.1cm/[r]^u & \mathrm{F} \ar@/^0.1cm/[l]^v}\neq\xymatrix@1{0\ar@/^0.1cm/[r] & 0 \ar@/^0.1cm/[l]}$
is isomorphic to its minimal extension,
 i.e. $\xymatrix@1{\mathrm{E} \ar@/^0.1cm/[r]^-{v\circ u} & \mathrm{im}(v\circ u) \ar@/^0.1cm/[l]^-{inc}}$, then
$$\dim\mathrm{Z}(v\circ u)-\dim\mathrm{Z}(u\circ v)=\left(\dim\ker(v\circ u)\right)^2.$$
\end{prop}
\begin{proof}
Cf. \cite{paiva} Proposition 2.4.10. Idea: decompose in Jordan blocks.
\end{proof}

To find a relationship between the terms figuring in the rigidity index in theorems \ref{jsnljAsi:ZNXakjs} and
\ref{rig:transforme:fourier}, take the representants of
$\mathcal{M}_{x_i}$, $x_i\in\Sigma$
in the category $\Theta$
\[
\xymatrix{
\mathrm{E}_i \ar@(dl,ul)[]^{\mathrm{T}_{i}-\bun } \ar@/^/[r]^{u_i} & \mathrm{F}_i \ar@/^/[l]^{v_i} \ar@(ur,dr)[]^{\mathrm{T}_{\mathrm{F}_i}-\bun. }
}
\]
\noindent Since $\mathrm{M}$ is irreducible, $\mathrm{M}$ is equal to its minimal extension. By Theorem
\ref{pervers:classification} {\romannumeral 2}) and Definition \ref{asdnsjdnja},
for each $x_i\in\Sigma\cap\mathbb{C}$, the pair above is equivalent to
\[
\xymatrix{
\mathrm{E}_i \ar@(dl,ul)[]^{\mathrm{T}_{i}-\bun } \ar@{->>}@/^/[r]^{\mathrm{T}_{i}-\bun } & \mathrm{F}_i \ar@{_{(}->}@/^/[l]^{inc} \ar@(ur,dr)[]^{\mathrm{T}_{\mathrm{F}_i}-\bun. }
}
\]
\noindent 
As $\mathrm{M}$ is holonomic, $\mathrm{M}=A_1/I$. In particular if one takes
a division basis $(P_p,\dots,P_q)$ of $I$, then $\dim\mathrm{E}_i=\deg_{\partial_x}P_p$. A proof
can be found in \cite{le} Theorem I.1.1. Without loss of generality
$\mathrm{E}_i=\mathrm{E}$, so the pairs above can be rewritten as follows
\begin{equation}
\xymatrix{
\mathrm{E} \ar@(dl,ul)[]^{\mathrm{T}_{i}-\bun } \ar@{->>}@/^/[r]^{\mathrm{T}_{i}-\bun } & \mathrm{F}_i \ar@{_{(}->}@/^/[l]^{inc} \ar@(ur,dr)[]^{\mathrm{T}_{\mathrm{F}_i}-\bun. }
}
\end{equation}
On the other hand, as ${\mathrm{M}}_{\mathcal{F}}$ is regular at 0, $({\mathcal{M}}_{\mathcal{F}})_0$ is equivalent to
\[
\xymatrix{
\hat{\mathrm{E}} \ar@(dl,ul)[]^{\hat{\mathrm{T}}_{}-\bun } \ar@/^/[r]^{u} & \hat{\mathrm{F}} \ar@/^/[l]^{v} \ar@(ur,dr)[]^{\mathrm{T}_{\hat{\mathrm{F}}}-\bun. }
}
\]
\noindent As $\mathrm{M}$ is irreducible, ${\mathrm{M}}_{\mathcal{F}}$
is irreducible (because the Fourier transform is an equivalence of categories), therefore coincides
with its minimal extension, so by Theorem \ref{prop:min:pervers} and Definition \ref{asdnsjdnja} the pair above is equivalent to
\begin{equation}
\xymatrix{
\hat{\mathrm{E}} \ar@(dl,ul)[]^{\hat{\mathrm{T}}_{}-\bun } \ar@{->>}@/^/[r]^{\hat{\mathrm{T}}-\bun } & \hat{\mathrm{F}} \ar@{_{(}->}@/^/[l]^{inc} \ar@(ur,dr)[]^{\mathrm{T}_{\hat{\mathrm{F}}}-\bun. }
}
\end{equation}
\noindent
Besides there is also the information provided by the Turrittin decomposition of the Fourier transform at infinity.
Let $(\hat{\mathcal{N}},\hat{\nabla})$
be the formalized of the meromorphic connexion $(\mathcal{N},\nabla)\doteq{\mathcal{M}}_{\mathcal{F}}[*\{\infty\}]$.
Its Turrittin~decomposition~is
$$(\hat{\mathcal{N}},\hat{\nabla})\simeq\bigoplus_{i=1}^k\hat{\mathcal{E}}^{\varphi_i}\otimes(\hat{\mathrm{R}}_i,\hat{\nabla}_i).$$
\noindent By Remark \ref{hbbcjbclkjbc},
each connexion $(\hat{\mathrm{R}}_i,\hat{\nabla}_i)$ is equivalent to
\[
\xymatrix{
\hat{\mathrm{F}}_i \ar@{->}@/^/[r]^-{\mathrm{T}_{\hat{\mathrm{F}}_i}-\bun } & \hat{\mathrm{F}}_i \ar@{=}@/^/[l]^-{\bun }.
}
\]

In \cite{mal} Malgrange shows analytically that $\mathrm{F}_i=\hat{\mathrm{F}}_i$ and $\hat{\mathrm{F}}=\mathrm{E}_{\infty}$, cf. \cite{mal}
Theorem XII.2.9. Moreover the later equality is a corollary of this stronger result.

\begin{lem}\label{hkbhlbfdxeraswa}
Let $\mathcal{M}$ be a regular holonomic $\mathcal{D}_{\mathbb{P}^1}[*\{\infty\}]$-module and 
$\mathcal{M}_{\mathcal{F}}$ its Fourier transform.
If $(\mathrm{E}',\mathrm{F}')$ (resp. $(\hat{E},\hat{F})$) is the pair of vector spaces equivalent
to $\mathcal{M}_{\infty}$ (resp. $\left(\mathcal{M}_{\mathcal{F}}\right)_0$), then $\hat{F}=F'$ and
$\mathrm{T}_{\hat{F}}=\mathrm{T}_{F'}$.
\end{lem}

\begin{proof}
An algebraic proof can be found in \cite{paiva} Lemma 2.6.22.
\end{proof}

\noindent Given all this one is led to believe that the following two lemmas are true and one 
proves that algebraically.

\begin{lem}\label{fdffdjrdfjtgvjj}
If $\mathcal{M}_{\mathrm{min}}=\mathcal{M}$, then for each
$x_i\in\Sigma\cap\mathbb{C}$ the monodromies
$\mathrm{T}_{\mathrm{F}_i}:\mathrm{F}_i\rightarrow\mathrm{F}_i$
and $\mathrm{T}_{\hat{\mathrm{F}}_i}:\hat{\mathrm{F}}_i\rightarrow\hat{\mathrm{F}}_i$ are conjugated,
i.e. $\mathcal{M}_{x_i}$ is equivalent to
\[
\xymatrix{
\mathrm{E} \ar@(dl,ul)[]^{\mathrm{T}_{i}-\bun } \ar@{->>}@/^/[r]^{\mathrm{T}_{i}-\bun } & \mathrm{F}_i \ar@{_{(}->}@/^/[l]^{inc} \ar@(ur,dr)[]^{\mathrm{T}_{\hat{\mathrm{F}}_i}-\bun. }
}\]
\end{lem}

\begin{lem}\label{sdusncdlsfdxjfsdjcsud}
If $({\mathcal{M}}_{\mathcal{F}})_{\mathrm{min}}={\mathcal{M}}_{\mathcal{F}}$ then
the monodromies at $\tau=0$
$\mathrm{T}_{\hat{\mathrm{E}}}:\hat{\mathrm{E}}\rightarrow\hat{\mathrm{E}}$
and $\mathrm{T}_{\infty}:\mathrm{F}_{\infty}\rightarrow\mathrm{F}_{\infty}$ are conjugated,
i.e. $({\mathcal{M}}_{\mathcal{F}})_{0}$ is equivalent to
\[
\xymatrix{
\hat{\mathrm{E}} \ar@(dl,ul)[]^{\hat{\mathrm{T}}_{0}-\bun } \ar@{->>}@/^/[r]^{\mathrm{T}_{0}-\bun } & \hat{\mathrm{F}} \ar@{_{(}->}@/^/[l]^{inc} \ar@(ur,dr)[]^{\mathrm{T}_{\infty}-\bun. }
}
\]
\end{lem}

To prove Lemma \ref{fdffdjrdfjtgvjj}, one briefly recalls the notation used by Sabbah in \cite{cs2} 
to compute the Turrittin decomposition of $(\hat{\mathcal{N}},\hat{\nabla})$ in the microlocalization context.
Those notations will be used just
in the following three statements.

\begin{note}
The Fourier transform is the isomorphism of algebras:
\begin{eqnarray*}
\mathbb{C}[t]\langle\partial_t\rangle & \longrightarrow & \mathbb{C}[\tau']\langle\partial_{\tau'}\rangle\\
t & \longmapsto & -\partial_{\tau'}\\
\partial_t & \longmapsto & \tau'
\end{eqnarray*}
\noindent denoted by $P\mapsto \hat{P}$. Every $\mathbb{C}[t]\langle\partial_t\rangle$-module $\mathrm{M}$
becomes this way a $\mathbb{C}[\tau']\langle\partial_{\tau'}\rangle$-module denoted by $\hat{\mathrm{M}}$
and called the Fourier transform of $\mathrm{M}$. From now on $\mathrm{M}$ is assumed holonomic with regular
singularities including at infinity. One knows that $\mathrm{M}$ only has a regular singularity at $\tau'=0$
and a possibly irregular singularity at $\tau'=\infty$. The localized module $\hat{\mathrm{M}}[\tau']$ is still
holonomic and is a free $\mathbb{C}[\tau',\tau^{\prime-1}]$-module of finite rank. It can be regarded as
a meromorphic bundle on the Riemann sphere $\mathbb{P}^1$, covered by charts with coordinates $\tau$ and
$\tau'$ and with the transition map $\tau=\tau^{\prime-1}$ on their intersection. Thus it is a free
$\mathbb{C}[\tau,\tau^{-1}]$-module. Finally one denotes by $\mathcal{M}$ the $\mathcal{D}_{\mathbb{P}^1}$-module
$\mathcal{O}_{\mathbb{P}^1}\otimes_{\mathbb{C}[t]}\mathrm{M}$.
\end{note}

\begin{lem}
\emph{(\cite{cs2}, Lemma 3.3)}
\label{euhdlutfdndkftrud}
The microlocalized module $\mathcal{M}^{\mu}$ has support in the set of singular points of $\mathcal{M}$.
\end{lem}

\noindent As $\mathcal{M}^{\mu}$ has support at the singular points $c$ of $\mathcal{M}$,
one can write $\Gamma(\mathbb{C},\mathcal{M}^{\mu})=\oplus_c\mathcal{M}^{\mu}_c$.

\begin{prop}\label{koegkdfkgfgkbfkgvff}
\emph{(\cite{cs2}, Proposition 3.4)}
At any singular point $c$ of $\mathcal{M}$, the germ
$\mathcal{E}^{c/\tau}\otimes\mathcal{M}^{\mu}_c$ is a
$(\hat{k}\doteq\mathbb{C}[\![\tau]\!][\tau^{-1}],\hat{\nabla})$-vector space with regular
singularities.
\end{prop}

\begin{prop}
\emph{(\cite{cs2}, Proposition 3.6)}
\label{qjbsdlsjcnsdjsjcsljc}
The composed $\mathbb{C}[\![\tau]\!]$-linear mapping
$$\mathbb{G}^{\hat{}}\doteq\hat{k}\otimes_{\mathbb{C}[\tau^{-1}]}\hat{\mathrm{M}}\rightarrow\Gamma(\mathbb{C},\hat{k}\otimes_{\mathbbm{C}[\partial_t]}\mathcal{M})\rightarrow\Gamma(\mathbb{C},\mathcal{M}^{\mu})$$
\noindent is an isomorphism.
\end{prop}

\begin{proof}[Proof of Lemma \ref{fdffdjrdfjtgvjj}]
Thanks to Lemma \ref{euhdlutfdndkftrud}, for each $x_i\in\Sigma\cap\mathbb{C}$ the microlocalization morphism $\mu$
allows the construction of the exact sequence of holonomic $\mathcal{D}_{x_i}$-modules
\begin{equation}\label{dsdcdjgfcjggcj}
\xymatrix{
0 \ar[r] & \ker\mu \ar@{^{(}->}[r] & \mathcal{M}_{x_i} \ar@{->}[r]^-{\mu} & (\mathcal{M}_{x_i})^{\mu} \ar[r] & \mathrm{coker}\mu \ar[r] & 0.
}
\end{equation}

\noindent As $\ker\mu\simeq\mathcal{D}_{x_i}/\mathcal{D}_{x_i}.(\partial_{x_i})^k$
(resp. $\mathrm{coker}\mu\simeq\mathcal{D}_{x_i}/\mathcal{D}_{x_i}.(\partial_{x_i})^{k'}$)
for a given $k\in\mathbb{N}$ (resp. $k'\in\mathbb{N}$ ), $\ker\mu$ (resp. $\mathrm{coker}\mu$) is equivalent to
$\xymatrix{\mathbb{C}^k \ar@/^/[r] & 0 \ar@/^/[l]}$ (resp. $\xymatrix{\mathbb{C}^{k'} \ar@/^/[r] & 0 \ar@/^/[l]}$).
By Proposition \ref{koegkdfkgfgkbfkgvff} $(\mathcal{M}_{x_i})^{\mu}$ is a regular meromorphic connexion
therefore Remark \ref{hbbcjbclkjbc} implies that it is equivalent to
\[
\xymatrix{
\hat{\mathrm{F}}_i \ar@{->}@/^/[r]^-{\mathrm{T}_{\hat{\mathrm{F}}_i}-\bun } & \hat{\mathrm{F}}_i \ar@{=}@/^/[l]^-{\bun }.
}
\]
\noindent These equivalences imply that the exact sequence (\ref{dsdcdjgfcjggcj}) is equivalent to the exact
sequence of pairs of vector spaces
\[
\xymatrix @R=2em @C=3.5em {
0 \ar[r] & \mathbb{C}^k \ar@{^{(}->}[r] \ar@/^/[d] & \mathrm{E} \ar@/^/@{->>}[d]^{ \mathrm{T}_i-\bun } \ar@{->}[rr]^{\alpha} && \hat{\mathrm{F}}_i \ar@/^/@{->}[d]^{\mathrm{T}_{\hat{\mathrm{F}}_i}-\bun } \ar[r] & \mathbb{C}^{k'} \ar[r] \ar@/^/[d] & 0\\
0 \ar[r] & 0 \ar[r] \ar@/^/[u] & \mathrm{F}_i \ar@{_{(}->}@/^/[u]^{j} \ar[rr]_{\beta}^{\sim} && \hat{\mathrm{F}}_i \ar@{=}@/^/[u]^{ \bun } \ar[r] & 0 \ar[r] \ar@/^/[u] & 0.
}
\]
\noindent Since the rows are exact, $\beta$ is an isomorphism, furthermore the commutativity of this diagram
implies that $\alpha\circ j=\beta$ and $\beta\circ(\mathrm{T}_i-\bun )= (\mathrm{T}_{\hat{\mathrm{F}}_i}-\bun )\circ\alpha,$ 
thus $\beta\circ(\mathrm{T}_i-\bun )\circ j= (\mathrm{T}_{\hat{\mathrm{F}}_i}-\bun )\circ\alpha\circ j,$ hence $\beta\circ\mathrm{T}_i=\mathrm{T}_{\hat{\mathrm{F}}_i}\circ\beta$.
\end{proof}

\begin{coro}\label{ijfjffjkgfigkid}
$\dim\mathrm{Z}(\mathrm{T}_i)-\dim\mathrm{Z}(\mathrm{T}_{\hat{\mathrm{F}}_i})=\left(\dim\ker(\mathrm{T}_i-\bun )\right)^2=(\dim\mathrm{E}-\dim\hat{\mathrm{F}}_i)^2$
\end{coro}

\begin{proof}
As the pair of vector spaces $(\mathrm{E},\mathrm{F}_i,\mathrm{T}_i-\bun,j)$ is minimal, it follows from
Proposition \ref{jbcbjkbjzka} that
$\dim\mathrm{Z}(\mathrm{T}_i)-\dim\mathrm{Z}(\mathrm{T}_{\hat{\mathrm{F}}_i})=\left(\dim\ker(\mathrm{T}_i-\bun)\right)^2$.
On the other hand 
$\dim\ker(\mathrm{T}_i-\bun )+\dim\mathrm{im}(\mathrm{T}_i-\bun )=\dim\mathrm{E}$.
Since
 $\mathrm{T}_{i}-\bun :\mathrm{E}\rightarrow\mathrm{F}_i$ is onto
$\dim\mathrm{F}_i=\dim\mathrm{im}(\mathrm{T}_i-\bun),$ 
 $\left(\dim\ker(\mathrm{T}_i-\bun)\right)^2=(\dim\mathrm{E}-\dim\hat{\mathrm{F}}_i)^2$.
\end{proof}

\begin{proof}[Proof of Lemma \ref{sdusncdlsfdxjfsdjcsud}]
Let us start with the exact sequence
\[
\xymatrix{
0 \ar[r] & \ker\lambda \ar@{^{(}->}[r] & \mathcal{M}_{\infty} \ar@{->}[r]^-{\lambda} & \mathcal{M}_{\infty}[*\{\infty\}] \ar[r] &  \mathrm{coker}\;\lambda \ar[r] & 0.
}
\]
\noindent As $\ker\lambda\simeq\mathcal{D}_{t'}/\mathcal{D}_{t'}t^{\prime k}$
(resp. $\ker\lambda\simeq\mathcal{D}_{t'}/\mathcal{D}_{t'}t^{\prime k'}$)
for a given $k\in\mathbb{N}$ (resp. $k'\in\mathbb{N}$), the exact sequence above gives rise
to the exact sequence
\[
\xymatrix @C=1.95em {
0 \ar[r] & \mathcal{D}_{\tau}/\mathcal{D}_{\tau}(\partial_{\tau})^k \ar@{^{(}->}[r] & ({\mathcal{M}}_{\mathcal{F}})_{0} \ar@{->}[r] & ({\mathcal{M}[t^{-1}]}_{\mathcal{F}})_{0} \ar[r] & \mathcal{D}_{\tau}/\mathcal{D}_{\tau}(\partial_{\tau})^{k'} \ar[r] & 0,
}
\]
\noindent which is equivalent to the exact sequence of pairs of vector spaces
\[
\xymatrix{
0 \ar[r] & \mathbb{C}^k \ar@{^{(}->}[rr] \ar@/^/[d] && \hat{\mathrm{E}} \ar@/^/@{->>}[d]^{\hat{\mathrm{T}}_0-\bun } \ar@{->}[rr]^{\alpha} && \widehat{\mathrm{E}'} \ar@/^/@{->}[d]^{u'} \ar[rr] && \mathbb{C}^{k'} \ar[r] \ar@/^/[d] & 0\\
0 \ar[r] & 0 \ar[rr] \ar@/^/[u] && \hat{\mathrm{F}} \ar@{_{(}->}@/^/[u]^j \ar[rr]_{\beta}^{\sim} && \widehat{\mathrm{F}'} \ar@{->}@/^/[u]^{v' } \ar[rr] && 0 \ar[r] \ar@/^/[u] & 0.
}
\]
\noindent Since rows are exact, $\beta$ is an isomorphism. Furthermore the commutativity of the diagram above 
implies that $u'\circ\alpha=\beta\circ(\hat{\mathrm{T}}_0-\bun)$ and $v'\circ\beta=\alpha\circ j,$
thus $u'\circ\alpha\circ j=\beta\circ(\hat{\mathrm{T}}_0-\bun)\circ j,$
hence
$(\mathrm{T}_{\widehat{\mathrm{F}'}}-\bun )\circ\beta=\beta\circ(\hat{\mathrm{T}}_0-\bun),$
$\beta\circ\hat{\mathrm{T}}_0=\mathrm{T}_{\widehat{\mathrm{F}'}}\circ\beta$. Thanks to Lemma \ref{hkbhlbfdxeraswa}
$\hat{\mathrm{T}}_0$ and $\mathrm{T}_{\infty}$ are conjugated.
\end{proof}

\begin{coro}\label{jvldbhvldajbfdlhvb}
$\dim\mathrm{Z}(\mathrm{T}_{\infty}) - \dim\mathrm{Z}(\hat{\mathrm{T}}_{0})=-(\dim\ker(\hat{\mathrm{T}}_0-\bun ))^2=-(\dim\hat{\mathrm{F}}-\dim\mathrm{F}_{\infty})^2$
\end{coro}

\begin{proof}
Similar to Corollary \ref{ijfjffjkgfigkid}.
\end{proof}

\begin{coro}\label{sbclhblslhdbcd}
$\dim\ker(\hat{\mathrm{T}}_0-\bun )=\dim\hat{\mathrm{E}}-\dim\mathrm{E}=
\sum_{i=1}^k\dim\hat{\mathrm{F}}_i-\dim\mathrm{E}$.
\end{coro}

\begin{proof}
Thanks to Lemma \ref{sdusncdlsfdxjfsdjcsud}
$\hat{\mathrm{T}}_0-\bun :\hat{\mathrm{E}}\rightarrow\hat{\mathrm{F}}$ is onto, therefore
$\dim\ker$ $(\hat{\mathrm{T}}_0-\bun )=\dim\hat{\mathrm{E}}-\dim\hat{\mathrm{F}}$. This Lemma also implies that 
$\dim\hat{\mathrm{F}}=\dim{F}_{\infty}=\dim{E}$, because $\mathcal{M}$ is regular and localized at infinity.
It follows from Proposition \ref{qjbsdlsjcnsdjsjcsljc} that
$\dim\ker(\hat{\mathrm{T}}_0-\bun )=\sum_{i=1}^k\dim\hat{\mathrm{F}}_i-\dim\hat{\mathrm{F}}$.
\end{proof}

\section{Rigidity index preservation}

The main result of this paper can now be proved.

\begin{teo}\label{sjcblJSBNCLKJSBC}
If $\mathcal{M}$ is a regular irreducible holonomic $\mathcal{D}_{\mathbb{P}^1}$-module with singularities on $\Sigma$ and localized at infinity, such that $\mathcal{M}_{min}\neq 0$, then the Fourier transform preserves the rigidity index.
\end{teo}

\begin{proof}
The irreducibility condition ensures that $\mathcal{M}_{min}=\mathcal{M}$, unless $\mathrm{H}_{[\Sigma]}(\mathcal{M})=\mathcal{M}$, cf. Proposition \ref{min}. It follows from Theorem \ref{rig:transforme:fourier} that:
$$\mathrm{rig}({\mathcal{M}}_{\mathcal{F}})=\dim\mathrm{Z}(\hat{\mathrm{T}}_0)+
\sum_{i=1}^k\dim\mathrm{Z}(\hat{\mathrm{T}}_i)+\sum_{i=1}^k(\dim\hat{\mathrm{F}}_i)^2-(\dim\hat{\mathrm{E}})^2.$$
\noindent Thanks to corollaries \ref{ijfjffjkgfigkid} and \ref{jvldbhvldajbfdlhvb} the equality above can be rewritten
as follows:
\begin{align*}
\mathrm{rig}({\mathcal{M}}_{\mathcal{F}}) & = 
\dim\mathrm{Z}(\mathrm{T}_{\infty})+(\dim\ker(\hat{\mathrm{T}}_0-\bun))^2+\\
 & +\sum_{i=1}^k\left[\dim\mathrm{Z}(\mathrm{T}_i)-(\dim\ker(\mathrm{T}_i-\bun ))^2\right]+\sum_{i=1}^k(\dim\hat{\mathrm{F}}_i)^2-(\dim\hat{\mathrm{E}})^2.\\
\end{align*}
\noindent Moreover by corollaries \ref{sbclhblslhdbcd} and \ref{ijfjffjkgfigkid}
\begin{align*}
\mathrm{rig}({\mathcal{M}}_{\mathcal{F}})
& = \dim\mathrm{Z}(\mathrm{T}_{\infty})+(\dim\hat{\mathrm{E}}-\dim\mathrm{E})^2+\\
& +\sum_{i=1}^k\left[\dim\mathrm{Z}(\mathrm{T}_i)-(\dim\mathrm{E}-\dim\hat{\mathrm{F}}_i)^2\right]  +\sum_{i=1}^k(\dim\hat{\mathrm{F}}_i)^2-(\dim\hat{\mathrm{E}})^2\\
& =  \dim\mathrm{Z}(\mathrm{T}_{\infty}) + \sum_{i=1}^k\dim\mathrm{Z}(\mathrm{T}_i)+(\dim\hat{\mathrm{E}}-\dim\mathrm{E})^2-\\
& -\sum_{i=1}^k\left[(\dim\mathrm{E})^2-2\dim\mathrm{E}\dim\hat{\mathrm{F}}_i+(\dim\hat{\mathrm{F}}_i)^2\right] +\sum_{i=1}^k(\dim\hat{\mathrm{F}}_i)^2-(\dim\hat{\mathrm{E}})^2\\
& =  \dim\mathrm{Z}(\mathrm{T}_{\infty}) + \sum_{i=1}^k\dim\mathrm{Z}(\mathrm{T}_i)+(\dim\hat{\mathrm{E}}-\dim\mathrm{E})^2-\\
& -k(\dim\mathrm{E})^2+2\dim\mathrm{E}\sum_{i=1}^k\dim\hat{\mathrm{F}}_i-(\dim\hat{\mathrm{E}})^2\\
& =  \dim\mathrm{Z}(\mathrm{T}_{\infty}) + \sum_{i=1}^k\dim\mathrm{Z}(\mathrm{T}_i)+(\dim\hat{\mathrm{E}})^2-2\dim\hat{\mathrm{E}}\dim\mathrm{E}+\\
&  +(\dim\mathrm{E})^2-k(\dim\mathrm{E})^2+2\dim\mathrm{E}\dim\hat{\mathrm{E}}-(\dim\hat{\mathrm{E}})^2\\
& =  (2-(k+1))(\dim\mathrm{E})^2+\dim\mathrm{Z}(\mathrm{T}_{\infty}) + \sum_{i=1}^k\dim\mathrm{Z}(\mathrm{T}_i)\\
& =  \mathrm{rig}(\mathcal{M}),
\end{align*}
\noindent therefore the Fourier transform preserves the rigidity index.
\end{proof}

\end{document}